\input ./style/arxiv-vmsta.cfg
\documentclass[numbers,compress,v1.0.1]{vmsta}
\usepackage{vtexbibtags}
\usepackage{authorquery}

\volume{5}
\issue{4}
\pubyear{2018}
\firstpage{429}
\lastpage{444}
\aid{VMSTA115}
\doi{10.15559/18-VMSTA115}
\articletype{research-article}

\startlocaldefs
\newcommand{\lleft}{\left}
\newcommand{\rright}{\right}

\newcommand{\rrvert}{\vert}
\newcommand{\llvert}{\vert}
\allowdisplaybreaks

\newtheorem{thm}{Theorem}
\newtheorem{lem}{Lemma}
\theoremstyle{definition}
\newtheorem{defin}{Definition}
\theoremstyle{remark}
\newtheorem{rem}{Remark}
\endlocaldefs

\begin{aqf}
\querytext{Q1}{Note: it seems that 'solution of equation' more often in yout article.}
\end{aqf}
\begin{document}

\begin{frontmatter}
\pretitle{Research Article}

\title{Approximation of solutions of the stochastic wave equation by
using the Fourier series}

\author{\inits{V.}\fnms{Vadym}~\snm{Radchenko}\thanksref{cor1}\ead[label=e1]{vradchenko@univ.kiev.ua}\orcid{0000-0002-8234-2797}}
\thankstext[type=corresp,id=cor1]{Corresponding author.}
\author{\inits{N.}\fnms{Nelia}~\snm{Stefans'ka}\ead[label=e2]{neliastefanska@gmail.com}}
\address{Department of Mathematical Analysis,
\institution{Taras Shevchenko National University of~Kyiv}, Kyiv 01601,
\cny{Ukraine}}



\markboth{V. Radchenko, N. Stefans'ka}{Approximation of solutions of the stochastic wave equation by
using the Fourier series}

\begin{abstract}
A one-dimensional stochastic wave equation driven by a general
stochastic measure is studied in this paper. The Fourier series
expansion of stochastic measures is considered. It is proved that
changing the integrator by the corresponding partial sums or by Fej\`{e}r
sums we obtain the approximations of mild solution of the equation.
\end{abstract}


\received{\sday{28} \smonth{5} \syear{2018}}
\revised{\sday{30} \smonth{8} \syear{2018}}
\accepted{\sday{30} \smonth{8} \syear{2018}}
\publishedonline{\sday{19} \smonth{9} \syear{2018}}
\end{frontmatter}

\section{Introduction}
\label{scintro}

In this paper we study the Cauchy problem for a one-dimensional
stochastic wave equation driven by a general stochastic measure. We
consider solution of this problem in the mild sense (see \eqref{mweq1}
below). Our goal is to show the convergence of solutions of equations
using the approximation of stochastic measures by partial sums of
Fourier series and Fej\`{e}r sums.

Existence and uniqueness of the solution of\querymark{Q1} our equation are obtained
in~\cite{Bodnarchuk}. The wave equation driven by stochastic measure
defined on subsets of the spatial variable was considered in~\cite
{BodnarchukWaveInR}. Similar equations driven by random stable noises
are studied in \cite{PryharaShevchenkoWeqR2,PryharaShevchenkoWeqR3}, where the properties of generalized solutions
are investigated.

Convergence of solutions of stochastic wave equation by using the
approximation of stochastic integrator was studied in \cite
{DelgadoVencesSanzSole,DelgadoVencesSanzSoleII}. Mild solutions
of equations driven by the Gaussian random field in dimension three
were considered in these papers.

Approximation of stochastic measures may be obtained by using Fourier
and Fourier--Haar series, the corresponding results are given in \cite
{radt18,RadStef}. The partial sums of the resulting series
generate random functions of sets, which are signed measures for each
fixed $\omega\in\varOmega$. The resulting equations can be solved as a
nonstochastic equation for each $\omega$. The results of our paper
imply that we obtain in this way an approximation of the solution
of~\eqref{mweq1}.

Papers \cite{radt18,RadStef} contain 
examples of applying
Fourier series of stochastic measures to the convergence of
solutions of the stochastic heat equation. A similar application of the
Fourier transform is given in~\cite{rad16}. Continuous dependence of
solutions of wave equation upon the data was studied in~\cite
{BodnarchukWaveInR,Bodnarchuk}. In this paper we obtain a
continuous dependence upon the values of stochastic integrator of the equation.

The paper is organized as follows. In Section~\ref{scprob} we recall
the basic facts about stochastic measures and Fourier series. Important
auxiliary lemma concerning convergence of stochastic integrals is
proved in Section~\ref{scdodt}. The formulation of the Cauchy problem
and theorem about approximation of the solution by using Fourier
partial sums are given in Section~\ref{scnbfrw}. The similar
approximation that uses Fej\`{e}r sums is obtained in Section~\ref
{scnbfew}. Section~\ref{scexam} contains one example with comments
about the convergence rate.

\section{Preliminaries}
\label{scprob}

Let ${\sf X}$ be an arbitrary set and let $\mathcal{B}({\sf X})$ be a
$\sigma$-algebra of subsets of ${\sf X}$.
Let ${\sf L}_0(\varOmega,\mathcal{F},{\sf P})$ be the set of (equivalence
classes of) all real-valued
random variables defined on a complete probability space $(\varOmega
,\mathcal{F},{\sf P})$. The convergence in ${\sf L}_0(\varOmega,\mathcal
{F},{\sf P})$ is understood to be in probability.

\begin{defin}
A $\sigma$-additive mapping $\mu:\ \mathcal{B}\to{\sf L}_0$ is called
a {\em stochastic measure} (SM).
\end{defin}

We do not assume positivity or moment existence for $\mu$. In other
words, $\mu$ is a vector measure with values in ${\sf L}_0$.

For a deterministic measurable function $g:{\sf X}\to\mathbb{R}$ and
SM~$\mu$, an integral of the form $\int_{\sf X}g\,d\mu$ is defined and
studied in~\cite[Chapter 7]{Kwapien}, see also~\cite[Chapter
1]{RadchenkoInt}. In particular, every bounded
measurable $g$ is integrable with respect to any~$\mu$. Moreover, an
analogue of the Lebesgue dominated convergence theorem holds for this
integral (see \cite[Theorem 7.1.1]{Kwapien}, \cite
[Corollary~1.2]{RadchenkoInt}).

Important examples of SMs are orthogonal stochastic measures, $\alpha
$-stable random measures defined on a $\sigma$-algebra for $\alpha\in
(0,1)\cup(1,2]$ (see \cite[Chapter 3]{samtaq}). Conditions under which
a process with independent increments generates an SM may be found in
\cite[Chapter 7 and 8]{Kwapien}.

Let the random series $\sum_{n\ge1}\xi_n$ converge unconditionally in
probability, and $m_n$ be real signed measures on~${\mathcal{B}}$,
$|{\sf m}_n(A)|\le1$. Set $\mu(A)=\sum_{n\ge1}\xi_n {\sf m}_n(A)$.
Convergence of this series in probability follows from~\cite
[Theorem~V.4.2]{vahtar}, and $\mu$ is a SM by \cite[Theorem 8.6]{dretop}.

Many examples of the SMs on the Borel subsets of $[0,T]$ may be given
by the Wiener-type integral
\begin{equation}
\label{eqmuax} \mu(A)=\int_{[0,T]} {\mathbf1}_A(t)
\,dX_t.
\end{equation}

We note the following cases of processes $X_t$ in~\eqref{eqmuax} that
generate SM.
\begin{enumerate}

\item\label{itmart} $X_t$~-- any square integrable martingale.

\item\label{itfrbr} $X_t=W_t^H$~-- the fractional Brownian motion with
Hurst index $H>1/2$, see Theorem~1.1~\cite{memiva}.

\item\label{itsfrb} $X_t=S_t^k$~-- the sub-fractional Brownian motion
for $k=H-1/2,\ 1/2<H<1$, see Theorem 3.2~(ii) and Remark 3.3~c) in~\cite
{tudor09}.

\item\label{itrose} $X_t=Z_H^k(t)$~-- the Hermite process, $1/2<H<1$,
$k\ge1$, see~\cite{tudor07}.
\end{enumerate}

We will give another example. Let $\zeta$ be an arbitrary SM defined on
Borel subsets of $[a,b]\subset{\mathbb R}$, function $h:[0,T]\times
[a,b]\to{\mathbb R}$ be such that $h(0,y)=0$, and
\begin{equation}
\label{eqconf} \big|h(t,y)-h(s,x)\big|\le L\bigl(|t-s|+|y-x|^\gamma\bigr),\quad
\gamma>1/2,\quad L\in {\mathbb R}.
\end{equation}
Then $h(\cdot,y)$ is absolutely continuous for each $y$, $ |\frac
{\partial h(t, y)}{\partial t} |\le L$ a.~e., and we can define SM
\begin{equation}
\label{eqmuet} \mu(A)=\int_{[a,b]} \,d\zeta(y) \int
_{A}\dfrac{\partial h(t,
y)}{\partial t}\,dt,\quad A\in\mathcal{B}\bigl([0,T]
\bigr),
\end{equation}
see details in~\cite[Section~3]{radt06}. Note that Theorem~1 of \cite
{radt06} implies that the process
\begin{equation}
\label{eqmutp} \mu_t=\mu\bigl((0,t]\bigr)=\int_{[a,b]} h(t, y)
\,d\zeta(y),\quad t\in[0,T],
\end{equation}
has a continuous version. In this case the process $X_t=\mu_t$ in~\eqref
{eqmuax} defines an SM.

Let $\mathcal{B}$ be a Borel $\sigma$-algebra on $(0, 1]$. For
arbitrary SM $\mu$ on $\mathcal{B}$ we consider the Fourier series in
the following sense.

Denote
\begin{align}
\xi_k&=\int_{(0,1]}\exp \{-2\pi i k t
\}\,d\mu(t)\nonumber\\
&:=\int_{(0,1]} \cos( 2\pi k t)\,d\mu(t)- i \int
_{(0,1]} \sin(2\pi k t)\, d\mu(t), \quad k\in\mathbb{Z}.\label{koef}
\end{align}

\begin{defin}
The series
\begin{equation}
\label{fyr} \sum_{k\in\mathbb{Z}}\xi_k\exp \{2\pi
i k t \}
\end{equation}
is called the \emph{Fourier series of SM $\mu$}. The random variables
$\xi_k$ are called the \emph{Fourier coefficients of series~\eqref
{fyr}}. \emph{Partial sums of series~\eqref{fyr}} are given by
\begin{equation*}
S_{j}(t)=\sum_{|k|\le j} \xi_k
\exp \{2\pi i k t \},\quad j\in \mathbb{Z}_+,\quad t\in[0,1].
\end{equation*}
\end{defin}

Stochastic integrals on the right hand side of~\eqref{koef} are defined
for any $\mu$, since the integral functions are bounded.
Thus the Fourier series is well defined for every SM on $\mathcal{B}$.

We will also consider Fej\`{e}r sums for SM $\mu$:
\[
\tilde{S}_j(t)=\frac{1}{j+1}\sum_{0\le k\le j}S_{k}(t).
\]

For necessary information concerning the classical Fourier series, we
refer to~\cite{zygmund}. In the sequel, $C$ and
$C(\omega)$ will denote nonrandom and random constants respectively
whose exact value is not essential.

\section{Convergence of integrals}\label{scdodt}

Put
\[
\Delta_{kn}=\bigl((k-1)2^{-n}, k2^{-n}\bigr], \quad n\ge0,
\quad1\le k\le2^n.
\]

Let the function $g(z,s):Z\times[0,1]\rightarrow\mathbb{R}$
be such that $\forall z\in Z:\;g(z,\cdot)$ is continuous on $[0,1]$.
Here $Z=Z_0\times[0,1],\,Z_0$ is an arbitrary set, $z=(z_0,t)$.
Denote
\[
g^{(n)}(z,s)=g(z,0){\mathbf{1}}_{\{0\}}(s)+ \sum
_{1\le k\le
2^n}g\bigl(z,(k-1)2^{-n}\wedge t\bigr){
\mathbf{1}}_{\Delta_{kn}}(s).
\]

From \cite[Lemma 3]{RadchenkoHilbertSpace} it follows that the random function
\[
\eta(z,t)=\int_{(0,t]}g(z,s)d\mu(s),\quad z\in Z,
\]
has a version
\begin{equation}
\label{version} %
\begin{split} \widetilde{\eta}(z,t) & = \int
_{(0,t]}g^{(0)}(z,s)d\mu(s)
\\
&\quad + \sum_{n\ge1} \biggl(\int_{(0,t]}g^{(n)}(z,s)d
\mu(s)- \int_{(0,t]}g^{(n-1)}(z,s)d\mu(s) \biggr),
\end{split} %
\end{equation}
such that for all $\varepsilon>0,\,\omega\in\varOmega,\,z\in Z$
\begin{equation}
\begin{split}\label{eqestint} \bigl\llvert \widetilde{\eta}(z,t) \bigr
\rrvert & \le \bigl\llvert g(z,0) \mu\bigl((0,t]\bigr) \bigr\rrvert
\\
&\quad+ \biggl\{\sum_{n\ge1}2^{n\varepsilon}\sum
_{1\le k\le2^n} \bigl\llvert g\bigl(z,k2^{-n}\wedge t\bigr) -g
\bigl(z,(k-1)2^{-n}\wedge t\bigr) \bigr\rrvert ^2 \biggr\}^{\frac{1}{2}}
\\
&\quad\times \biggl\{\sum_{n\ge1}2^{-n\varepsilon}\sum
_{1\le k\le2^n} \bigl\llvert \mu\big(\Delta_{kn}
\cap(0,t] \big) \bigr\rrvert ^2 \biggr\}^{\frac{1}{2}}.
\end{split}
\end{equation}
We note that the series with the values of SM in \eqref{eqestint}
converges a. s. (see \cite[Lemma 3.1]{radt09}).

\begin{lem}\label{lmgjtg} Let $Z$ be an arbitrary set, $g,\ g_j:Z\times
[0,1]\to\mathbb{R}$, and the following conditions hold

(i) $\sup_{z\in Z,t\in[0,1]}|g_j(z,t)-g(z,t)|\to0$, $j\to\infty$;

(ii) for some constants $L_g>0$, $\beta(g)>1/2$
\[
\sup_{z\in Z,j\ge1}|g_j(z,t)-g_j(z,s)|\le
L_g |t-s|^{\beta(g)}\quad t,\ s\in[0,1];
\]

(iii) for some random constant $C_{\mu}(\omega)$
\[
|\mu((0,t])|\le C_{\mu}(\omega),\quad t\in(0,1].
\]
Then for versions~\eqref{version} of the processes
\[
\eta_j(z,t)=\int_{(0,t]}g_j(z,s)\,d
\mu(s),\qquad\eta(z,t)=\int_{(0,t]}g(z,s)\,d\mu(s)
\]
the following holds:
\begin{equation*}
\sup_{z\in Z,t\in[0,1]}\big|\eta_j(z,t)-\eta(z,t)\big|\to0\quad \textrm{a.
\ s.},\quad j\to\infty.
\end{equation*}
\end{lem}

\begin{proof} Without loss of generality, we can assume that $g=0$. For
$\eta_j$ we will use inequality~\eqref{eqestint}. Separating for each
$n$ intervals $\Delta_{kn}$ that contain $t$ and using the condition
(iii), we have
\begin{align*}
&\sum_{n\ge1}2^{-n\varepsilon}\sum
_{1\le k\le2^n} \bigl\llvert \mu \bigl(\Delta_{kn}\cap(0,t] \bigr) \bigr
\rrvert ^2
\\
&\quad\le\sum_{n\ge1}2^{-n\varepsilon}\sum
_{1\le k\le2^n} \bigl\llvert \mu (\Delta_{kn} ) \bigr\rrvert
^2 +\sum_{n\ge
1}2^{-n\varepsilon}
\bigl(2C_{\mu}(\omega)\bigr)^2.
\end{align*}
Since both series in the right hand side are finite a. s., we obtain a
random upper bound uniformly in $t$.
Condition (ii) implies that
\begin{equation}
\label{eqhjkn} \sum_{1\le k\le2^n} \bigl\llvert
g_j\bigl(z,k2^{-n}\wedge t\bigr) -g_j
\bigl(z,(k-1)2^{-n}\wedge t\bigr) \bigr\rrvert ^2
\le2^n L_{g}2^{-2n\beta(g)}.
\end{equation}
Let $\sup_{z\in Z,t\in[0,1]}|g_j(z,t)|=\delta_j$, $\delta_j\to0$. Then
\begin{equation}
\label{eqhjdj} \sum_{1\le k\le2^n} \bigl\llvert
g_j\bigl(z,k2^{-n}\wedge t\bigr) -g_j
\bigl(z,(k-1)2^{-n}\wedge t\bigr) \bigr\rrvert ^2
\le2^n\cdot4 \delta_j^2.
\end{equation}
The product of \eqref{eqhjkn} to the power $\theta$ and \eqref{eqhjdj}
to the power $1-\theta$ now satisfies
\[
\sum_{1\le k\le2^n} \bigl\llvert g_j
\bigl(z,k2^{-n}\wedge t\bigr) -g_j\bigl(z,(k-1)2^{-n}
\wedge t\bigr) \bigr\rrvert ^2\le C 2^n 2^{-2n\theta\beta
(g)}
\delta_j^{2(1-\theta)}.
\]
For $\frac{1}{2\beta(h)}<\theta<1$, $0<\varepsilon<2\theta\beta(h)-1$
we have
\[
\sup_{z\in Z,t\in[0,1]}\sum_{n\ge1}2^{n\varepsilon}
\sum_{1\le k\le
2^n} \bigl\llvert g_j
\bigl(z,k2^{-n}\wedge t\bigr) -g_j\bigl(z,(k-1)2^{-n}
\wedge t\bigr) \bigr\rrvert ^2\le C \delta_j^{2(1-\theta)}.
\]
The right hand side of the inequality tends to zero as $j\to\infty$. Since
\[
\sup_{z\in Z,t\in[0,1],j\ge1} \bigl\llvert g_j(z,0) \mu\bigl((0,t]\bigr) \bigr
\rrvert \to 0,\quad j\to\infty,
\]
applying~\eqref{eqestint} completes the proof of the lemma.
\end{proof}

\section{Approximation of solutions by using the Fourier partial
sums}\label{scnbfrw}

Consider the Cauchy problem for a one-dimensional stochastic wave equation
\begin{equation}
\label{weq} \lleft\{ %
\begin{array}{l}
\dfrac{\partial^2 u(t,x)}{\partial t^2}
=a^2\dfrac{\partial^2 u(t,x)}{\partial x^2}
+f(t,x,u(t,x))+ \sigma(t,x)\,\dot{\mu}(t) ,\\
u(0,x)=u_0(x),\quad\dfrac{\partial u(0,x)}{\partial t}=v_0(x),
\end{array} %
\rright.
\end{equation}
where $(t,x)\in[0,1]\times{\mathbb{R}}$, $a>0$, $\mu$ is an SM defined
on the Borel $\sigma$-algebra $\mathcal{B}((0, 1])$.

The solution of equation \eqref{weq} is understood in the mild sense,
\begin{equation}
\begin{split}\label{mweq1} u(t,x) & =\frac{1}{2}
\bigl(u_0(x+at)-u_0(x-at)\bigr)+ \frac{1}{2a}\int
_{x-at}^{x+at} v_0(y)\,dy
\\
&\quad +\frac{1}{2a}\int_0^t\,ds\int
_{x-a(t-s)}^{x+a(t-s)} f\bigl(s,y, u(s,y)\bigr)\, dy
\\
& \quad+\frac{1}{2a}\int_{(0,t]}\,d\mu(s) \int
_{x-a(t-s)}^{x+a(t-s)}\sigma(s,y)\,dy\,. \end{split} %
\end{equation}
The integrals of random functions with respect to $dx$ are taken for
each fixed $\omega\in\varOmega$.
We impose the following assumptions.

\textbf{A1.} Functions $u_0(y)=u_0(y,
\omega):{\mathbb R}\times\varOmega\to{\mathbb R}$ and $v_0(y)=v_0(y,
\omega):{\mathbb R}\times\varOmega\to{\mathbb R}$ are measurable and
bounded for every fixed $\omega\in\varOmega$.

\textbf{A2.} The function $f(s, y, v):[0,1]\times{{\mathbb R}}\times
{\mathbb R}\to{\mathbb R}$ is measurable and bounded.

\textbf{A3.} The function $f(s, y, v)$ is uniformly Lipschitz in $y,
v\in\mathbb R$:
\[
\bigl\llvert f(s, y_1, v_1)-f(s, y_2,
v_2) \bigr\rrvert \le L_f \bigl( \llvert
y_1-y_2 \rrvert + \llvert v_1-v_2
\rrvert \bigr)\, .
\]

\textbf{A4.} The function $\sigma(s,y):[0, 1]\times{\mathbb R}\to
{\mathbb R}$ is measurable and bounded.

\textbf{A5.} The function $\sigma(s,y)$ is H\"{o}lder continuous:
\[
\bigl\llvert \sigma(s_1,y_1)-\sigma(s_2,y_2)
\bigr\rrvert \le L_\sigma \bigl( \llvert s_1-s_2
\rrvert ^{\beta(\sigma)}+ \llvert y_1-y_2 \rrvert
^{\beta(\sigma)} \bigr),\quad 1/2<\beta(\sigma)\le1.
\]

\textbf{A6.} For some random constant $C_{\mu}(\omega)$ $|\mu
((0,t])|\le C_{\mu}(\omega),\ t\in(0,1]$.

From A1--A5 it follows that equation \eqref{mweq1} has a unique
solution (see Theorem 2.1~\cite{Bodnarchuk}). The H\"{o}lder continuity
condition was imposed on $u_0$ in \cite{Bodnarchuk}, but was not used
in proof of existence and uniqueness of the solution.

Note that the processes $X_t$ in examples~\ref{itfrbr}--\ref{itrose} of
SMs and $\mu_t$ in~\eqref{eqmutp} are continuous, therefore A6 is
fulfilled in these cases.

Consider the following equations:
\begin{equation}
\begin{split}\label{mweq1nf} u_j(t,x) & =
\frac{1}{2}\bigl(u_0(x+at)-u_0(x-at)\bigr)+
\frac{1}{2a}\int_{x-at}^{x+at} v_0(y)
\,dy
\\
&\quad +\frac{1}{2a}\int_0^t\,ds\int
_{x-a(t-s)}^{x+a(t-s)} f\bigl(s,y, u_j(s,y)\bigr)
\,dy
\\
&\quad +\frac{1}{2a}\int_{(0,t]}S_j(s)\,ds \int
_{x-a(t-s)}^{x+a(t-s)}\sigma(s,y)\,dy\,. \end{split} %
\end{equation}

\begin{thm}\label{thnabf} Let A1--A6 are fulfilled, and assume that
the following conditions hold: if $h\in{\sf L}_2( (0, 1] )$ then $h$
is integrable with respect to $\mu$, and
\begin{equation}
\label{eqlpmu} {\it if}\quad \int_{(0, 1] } \big|h_j(x)\big|^2
\,dx\to0,\quad j\to\infty\quad{\it then}\quad \int_{(0, 1] }
h_j(x)\,d\mu(x)\stackrel{\sf P} {\to} 0,\quad j\to\infty.
\end{equation}

Then $u$ from~\eqref{mweq1} and $u_j$ from~\eqref{mweq1nf} have
versions such that for every $0<\delta<1$
\begin{equation}
\label{eqsupu} \sup_{x\in{\mathbb R},t\in[0,1-\delta]} \bigl\llvert u_j(t,x)-u(t,x)
\bigr\rrvert \stackrel{\sf P} {\to} 0,\quad j\to\infty.
\end{equation}
\end{thm}

\begin{proof} The outline of the proof is the following. Denote
\[
g(t,x,s)=\int_{x-a(t-s)}^{x+a(t-s)}\sigma(s,y)\,dy,\quad0\le s
\le t.
\]
Step 1 -- using the Gronwall's inequality, we will estimate supremum in \eqref
{eqsupu} by the value
\[
\sup_{t,x} \left|\int_{(0,t]} g \,d\mu-\int
_{(0,t]} g S_j\,ds \right|.
\]
Further (Step 2), we consider the continuation of $g(t,x,s), 0\le s\le
t,$ to the function $g_{\delta}(t,x,s), 0\le s\le1$, and estimate
\[
\sup_{t,x} \left|\int_{(0,t]}g\,d\mu-\int
_{(0,1]} g_{\delta}\,d\mu \right|.
\]
In Step 3 we consider
\[
\sup_{t,x} \left|\int_{(0,1]} g_{\delta}\,d
\mu-\int_{(0,1]} g_{\delta
}S_j\,ds \right|,
\]
and in Step 4 we estimate
\[
\sup_{t,x} \left|\int_{(0,1]} g_{\delta}S_j
\,ds-\int_{(0,t]} g S_j\, ds \right|
\]
and make the concluding remarks.

\textbf{Step 1.} We will take versions~\eqref{version} for all
integrals with respect to $\mu$.
We have
\begin{align}
&\big|u(t,x)-u_j(t,x)\big|\nonumber
\\
&\quad\le\frac{1}{2a} \Bigg|\int_0^t\,ds\int
_{x-a(t-s)}^{x+a(t-s)} f\bigl(s,y, u(s,y)\bigr)\,dy\nonumber
\\
&\qquad -\int_0^t\,ds\int_{x-a(t-s)}^{x+a(t-s)}
f\bigl(s,y, u_j(s,y)\bigr)\,dy \Bigg|
\label{eqesuj} \\*
&\qquad+ \frac{1}{2a} \Bigg|\int_{(0,t]}g(t,x,s)\,d\mu(s) -\int
_{(0,t]} S_j(s)g(t,x,s)\,ds \Bigg|\nonumber
\\
&\quad \stackrel{A3} {\le} C \int_0^t\,ds \int
_{x-a(t-s)}^{x+a(t-s)} \big|u(s,y)-u_j(s,y)\big|\,dy\nonumber
\\
&\qquad +\frac{1}{2a} \Bigg|\int_{(0,t]}g(t,x,s)\,d\mu(s) -\int
_{(0,t]} S_j(s)g(t,x,s)\,ds\Bigg |. \nonumber
\end{align}\goodbreak

Denote
\begin{equation*}
\begin{split}\xi_j(t)&=\sup_{x\in{\mathbb R}}\big|u(t,x)-u_j(t,x)\big|,
\\
\eta_{\delta j}&= \frac{1}{2a}\sup_{x\in{\mathbb R},\ t\in[0,1-\delta
]} \left|\int
_{(0,t]}g(t,x,s)\,d\mu(s) -\int_{(0,t]}
S_j(s)g(t,x,s)\,ds \right|. \end{split} %
\end{equation*}
Then
\begin{equation}
\label{eqxijf} \xi_j(t)\le C \int_0^t
\xi_j(s)\,ds+\eta_{\delta j},\quad t\in [0,1-\delta].
\end{equation}
Applying the Gronwall's inequality, we get
\begin{equation}
\label{eqeste} \xi_j(t)\le\eta_{\delta j}+C\int
_0^t \exp\bigl\{C(t-s)\bigr\}
\eta_{\delta j}\, ds\le C\eta_{\delta j},\quad t\in[0,1-\delta].
\end{equation}
Further, we will estimate $\eta_{\delta j}$. In particular, we will get
that $\eta_{\delta j}<+\infty$~a.~s. From this, A2 and first inequality
in~\eqref{eqesuj} it follows that $\xi_j(t)\le C(\omega)<\infty$~a.~s.

\textbf{Step 2.}
Note that $g(t,x,t)=0$. We define the function
\[
g_{\delta}(t,x,s)=g(t,x,s),\quad 0\le s\le t,\quad g_{\delta}(t,x,s)=0,\quad
t\le s< 1-\delta.
\]
Put $g_{\delta}(t,x,1)=g(t,x,0)$ and extend $g_{\delta}(t,x,s)$ for
$s\in[1-\delta,1]$ in a linear way such that the function is
continuous on $[0,1]$. Also, $g_{\delta}(t,x,s)$ has a continuous
periodic extension on ${\mathbb R}$ in a variable $s$ for fixed $x\in
{\mathbb R},\ t\in[0,1)$.

First, consider
\[
A_{\delta}:=\sup_{t,x} \left|\int_{(0,t]}g(t,x,s)
\,d\mu(s) -\int_{(0,1]} g_{\delta}(t,x,s)\,d\mu(s) \right|.
\]
By the definition of the function $g_{\delta}$, we have
\begin{equation*}
\begin{split} &\int_{(0,1]}g_{\delta}(t,x,s)
\,d\mu(s)-\int_{(0,t]} g(t,x,s)\,d\mu (s)=\int
_{(1-\delta,1]} g_{\delta}(t,x,s) \,d\mu(s)
\\
&\quad=\int_{(1-\delta,1]} g_{\delta}(t,x,0)\frac{s+\delta-1}{\delta} \,d\mu
(s)= g_{\delta}(t,x,0) \int_{(1-\delta,1]} \frac{s+\delta-1}{\delta} \,
d\mu(s) \end{split} %
\end{equation*}

By the analogue of the Lebesgue theorem,
\begin{equation}
\label{eqdelg} \int_{(1-\delta,1]} \frac{s+\delta-1}{\delta} \,d\mu(s)
\stackrel{\sf P} {\to} \mu\bigl(\{1\}\bigr),\quad\delta\to0.
\end{equation}

Condition~\eqref{eqlpmu} provides that $\mu(\{1\})=0$ a. s., also
$|g_{\delta}(t,x,0) |\le C$. Therefore,
\begin{equation}
\label{eqadst} A_{\delta}\stackrel{\sf P} {\to} 0,\quad\delta\to0.
\end{equation}

\textbf{Step 3.}
Further, we will estimate
\[
B_{\delta j}:=\sup_{t,x} \left|\int_{(0,1]}
g_{\delta}(t,x,s) \,d\mu(s) -\int_{[0,1]}g_{\delta}(t,x,s)S_{j}(s)
\,ds \right|.
\]
We have
\begin{align}
&\int_{[0,1]}g_{\delta}(t,x,s)S_{j}(s)
\,ds = \int_{[0,1]}g_{\delta}(t,x,s) \biggl( \sum
_{|k|\le j} \xi_k \exp\{2\pi i k s\} \biggr)\,ds
\nonumber\\
&\quad=\int_{[0,1]}g_{\delta}(t,x,s) \biggl( \sum
_{|k|\le j}\exp\{2\pi i k s\} \int_{(0,1]} \exp\{-2
\pi i k r\}\,d\mu (r) \biggr)\,ds
\nonumber\\
&\quad=\sum_{|k|\le j}\int_{[0,1]}g_{\delta}(t,x,s)
\biggl( \exp\{2\pi i k s\} \int_{(0,1]} \exp\{-2\pi i k r\}\,d
\mu(r) \biggr)\,ds
\label{mlgdsj} \\
&\quad\stackrel{(*)} {=}\sum_{|k|\le j}\int
_{(0,1]} \exp\{-2\pi i k r\}\,d\mu(r)\int_{[0,1]}
\exp\{2\pi i k s\}g_{\delta}(t,x,s)\,ds
\nonumber\\
&\quad=\int_{(0,1]} \biggl(\sum_{|k|\le j}
\exp\{-2\pi i k r\}\int_{[0,1]}\exp \{2\pi i k s
\}g_{\delta}(t,x,s)\,ds \biggr)d\mu(r) . \nonumber
\end{align}\goodbreak
(We can change the order of integration in (*) due to Theorems 1 and
2~\cite{radpro}.) Partial sums of Fourier series of functions $g_{\delta
}(t,x,r)$ in variable $r$ are given by
\[
g_{\delta j}(t,x,r)=\sum_{|k|\le j}\exp\{-2\pi i k r
\}\int_{[0,1]}\exp\{ 2\pi i k s\}g_{\delta}(t,x,s)\,ds.
\]
We will demonstrate that for fixed $\delta$ for functions $g_{\delta
j}$, $g_{\delta}$ and $z=(t,x)$ the conditions of Lemma~\ref{lmgjtg} hold.

Consider the Fourier coefficients
\begin{equation*}
\begin{split} a_{-k}(t,x) & = \int_{[0,1]}
\exp\{2\pi i k s\}g_{\delta}(t,x,s)\,ds,
\\
a_{-k}'(t,x) & = \int_{[0,1]}\exp\{2\pi
i k s\}\dfrac{\partial g_{\delta
}(t,x,s)}{\partial s}\,ds=-2\pi ik a_{-k}(t,x). \end{split}
\end{equation*}
For any set of indices $\mathbb{M}\subset\mathbb{Z}\setminus\{0\}$ we have
\begin{equation}
\label{eqaktx} %
\begin{split} \sup_{t,x} \sum
_{k\in\mathbb{M}} |a_k| & = \dfrac{1}{2\pi}\sup
_{t,x} \sum_{k\in\mathbb{M}} \dfrac{|a_k'|}{|k|}
\le\dfrac{1}{2\pi}\sup_{t,x} \biggl( \sum
_{k\in\mathbb{M}} \big|a_k'\big|^2
\biggr)^{1/2} \biggl( \sum_{k\in
\mathbb{M}}
\dfrac{1}{k^2} \biggr)^{1/2}
\\
& \le\dfrac{1}{2\pi}\sup_{t,x} \left\|\dfrac{\partial g_{\delta
}(t,x,s)}{\partial s}
\right\|_{{\sf L}_2} \biggl(\sum_{k\in\mathbb{M}} \dfrac{1}{k^2}
\biggr)^{1/2}. \end{split} %
\end{equation}
Obviously, the supremum of ${\sf L}_2$-norms (taken in the variable
$s$) will be finite for fixed~$\delta$. Thus,
\[
\sup_{t,x,r}\big|g_{\delta j}(t,x,r)-g_{\delta l}(t,x,r)\big|\le
\sup_{t,x}\sum_{l<|k|\le j}|a_k|
\to0,\quad l,\ j\to\infty,
\]
and sequence $g_{\delta j}(t,x,r)$, $j\ge1$, converges uniformly in
$(t,x,r)$. Also, it is well known that for our piecewise smooth
function $g_{\delta}$ the pointwise convergence\break $g_{\delta j}(t,x,s)\to
g_{\delta}(t,x,s)$, $j\to\infty$, holds. Therefore, condition (i) of
Lemma~\ref{lmgjtg} is fullfiled.

Further, we will check condition (ii) for $\beta(g)\,{=}\,1$. Using the
periodicity of $g_{\delta}(t,x,s)$ in~$s$, for $\rho\in\mathbb{R}$ we obtain
\[
g_{\delta j}(t,x,r+\rho)=\sum_{|k|\le j}\exp\{-2\pi i
k r\}\int_{[0,1]}\exp\{2\pi i k s\}g_{\delta}(t,x,s+\rho)
\,ds.
\]
Therefore, $g_{\delta j}(t,x,r+\rho)-g_{\delta j}(t,x,r)$ are partial
sums of the Fourier series of the function $g_{\delta}(t,x,s+\rho
)-g_{\delta}(t,x,s)$. We can repeat the reasoning from~\eqref{eqaktx}
for $\mathbb{M}=\mathbb{Z}$ and
\[
a_{-k}=\int_{[0,1]}\exp\{2\pi i k s\}
\bigl(g_{\delta}(t,x,s+\rho)-g_{\delta
}(t,x,s)\bigr)\,ds.
\]
It is easy to see that
\[
\sup_{t,x} \left\|\dfrac{\partial(g_{\delta}(t,x,s+\rho)-g_{\delta
}(t,x,s))}{\partial s} \right\|_{{\sf L}_2}\le C\rho.
\]
Since
\[
\big|g_{\delta j}(t,x,r+\rho)-g_{\delta j}(t,x,r)\big|\le\sum
_{k\in\mathbb{Z}} |a_k| ,
\]
we get (ii).

Lemma~\ref{lmgjtg} implies that
\begin{equation}
\label{eqbdst} B_{\delta j}\to0\quad \textrm{a.\ s.},\quad j\to\infty,
\end{equation}
for each fixed $\delta$.

\textbf{Step 4.}
It remains to consider
\begin{equation}
\label{mlcdej} %
\begin{split} C_{\delta j}&:=\sup
_{t,x} \left|\int_{(0,1]}S_j(s)g_{\delta}(t,x,s)
\, ds-\int_{(0,t]}S_j(s)g(t,x,s)\,ds \right|
\\
&=\sup_{t,x} \left|\int_{(0,1]}S_j(s)g_{\delta}(t,x,s)
\,ds-\int_{(0,t]}S_j(s)g_{\delta}(t,x,s)\,ds \right|
\\
&=\sup_{t,x} \left|\int_{(1-\delta,1]}S_j(s)g_{\delta}(t,x,s)
\,ds \right|
\\
&=\sup_{t,x} \left|\int_{(1-\delta,1]}S_j(s)g(t,x,0)
\frac{s+\delta
-1}{\delta}\,ds \right|
\\
&= \sup_{t,x}\big|g(t,x,0)\big| \left|\int_{(1-\delta,1]}S_j(s)
\frac{s+\delta
-1}{\delta}\,ds \right|
\\
&\le C \left|\int_{(1-\delta,1]}S_j(s)\frac{s+\delta-1}{\delta}\,ds
\right|=: C |\tilde{C}_{\delta j} |. \end{split} %
\end{equation}
If we consider the function $h_{\delta}(s)=\frac{s+\delta-1}{\delta
}\mathbf{1}_{[1-\delta,1]}$ and its corresponding $j$-th Fourier sum
$h_{\delta j}(s)$, then, as in~\eqref{mlgdsj}, we have
\begin{equation}
\label{eqhdej} \int_{(1-\delta,1]}S_j(s)\frac{s+\delta-1}{\delta}
\,ds=\int_{(0,1]}h_{\delta j}(s)\,d\mu(s).
\end{equation}
By the standard properties of Fourier sums,
\[
h_{\delta j}(s)\to h_{\delta}(s), \quad j\to\infty,
\]
in ${\sf L}_2([0,1])$.
From~\eqref{eqlpmu} we get
\begin{equation}
\label{eqhdejd} \tilde{C}_{\delta j}=\int_{(0,1]}h_{\delta j}(s)
\,d\mu(s){\stackrel{\sf P} {\to}} \int_{(0,1]}h_{\delta}(s)
\,d\mu(s):=D_{\delta},\quad j\to\infty.
\end{equation}
We have already noticed in~\eqref{eqdelg} that
\begin{equation}
\label{eqddst} D_{\delta}{\stackrel{\sf P} {\to}} 0,\quad\delta\to0.
\end{equation}

Finally, we have
\begin{equation}
\label{eqeabc} \eta_{\delta j}\le A_{\delta}+B_{\delta j}+C_{\delta j}.
\end{equation}
In order to explain that $\eta_{\delta j}\stackrel{\sf P}{\to} 0$, we
will use the seminorm
\[
\|\eta\|=\sup\bigl\{\alpha: {\sf P}\bigl(|\eta|\ge\alpha\bigr)\ge\alpha\bigr\},
\]
that corresponds to the convergence in ${\sf L}_0$. If \eqref{eqsupu}
does not hold then
\begin{equation}
\label{eqetal} \|\eta_{\delta j}\|\ge\alpha_0
\end{equation}
for some $\delta,\alpha_0>0$ and infinitely many $j$.

We have
\begin{equation*}
\begin{split} \|\eta_{\delta j}\|&\stackrel{\eqref{eqeabc}} {\le}
\|A_{\delta}\|+\| B_{\delta j}\|+\|C_{\delta j}\|\stackrel{
\eqref{mlcdej}} {\le} \| A_{\delta}\|+\|B_{\delta j}\|+\|C
\tilde{C}_{\delta j}\|
\\
&\le\|A_{\delta}\|+\|B_{\delta j}\|+\big\|C (\tilde{C}_{\delta j}-D_{\delta
})
\big\|+\|C D_{\delta}\|. \end{split} %
\end{equation*}
From \eqref{eqbdst} and~\eqref{eqhdejd} it follows that for each $\delta$
\begin{equation*}
{\limsup_{j\to\infty}}\|\eta_{\delta j}\| \le\|A_{\delta}
\|+\|C D_{\delta}\|,
\end{equation*}
\eqref{eqadst} and \eqref{eqddst} imply that
\begin{equation*}
\lim_{\delta\to0}{\limsup_{j\to\infty}}\|
\eta_{\delta j}\|=0.
\end{equation*}
This contradicts to~\eqref{eqetal} (the reduction of $\delta$ given in
the formulation of the theorem reinforces the assertion).
\end{proof}

\begin{rem}\label{rmtone} Note that condition~\eqref{eqlpmu} holds for
examples of SMs~\ref{itfrbr} and \ref{itsfrb} (see \cite{memiva,tudor09}).
If~\eqref{eqlpmu} is fulfilled for SM $\zeta$ in~\eqref
{eqmuet} then it holds for $\mu$. This follows from the boundedness of
$\frac{\partial h(t, y)}{\partial t}$ and properties of the integral,
see~\cite[Chapter 1]{RadchenkoInt}.

In our proof condition~\eqref{eqlpmu} was used only for convergence
in~\eqref{eqhdejd} and for equality $\mu(\{1\})=0$ a. s. If for given
$h_{\delta j}$ and~$\mu$ these statements hold true then the general
condition~\eqref{eqlpmu} can be discarded.
\end{rem}

In the next section, we will demonstrate that replacing partial sums of
the Fourier series by the corresponding Fej\`{e}r sums, we can omit
condition~\eqref{eqlpmu}.

\section{Approximation of solutions by using the Fej\`{e}r sums}\label{scnbfew}

Consider the following equations that use the Fej\`{e}r sums $\tilde
{S}_j(s)$ of SM $\mu$:
\begin{equation}
\begin{split}\label{mweq1nfe} \tilde{u}_j(t,x) & =
\frac{1}{2}\bigl(u_0(x+at)-u_0(x-at)\bigr)+
\frac{1}{2a}\int_{x-at}^{x+at} v_0(y)
\,dy
\\
& \quad+\frac{1}{2a}\int_0^t\,ds\int
_{x-a(t-s)}^{x+a(t-s)} f\bigl(s,y, \tilde {u}_j(s,y)
\bigr)\,dy
\\
& \quad+\frac{1}{2a}\int_{(0,t]}\tilde{S}_j(s)
\,ds \int_{x-a(t-s)}^{x+a(t-s)}\sigma(s,y)\,dy\,. \end{split}
\end{equation}

We show that the functions $\tilde{u}_j$ also approximate the solution
$u$ of equation~\eqref{mweq1}. Here we impose weaker conditions on $\mu
$ than in Theorem~\ref{thnabf}.

\begin{thm} Let A1--A6 hold. Then $u$ from~\eqref{mweq1} and $\tilde
{u}_j$ from~\eqref{mweq1nfe} have versions such that for every $0<\delta<1$
\begin{equation}
\label{eqsupuf} \sup_{x\in\mathbb{R},t\in[0,1-\delta]} \bigl\llvert \tilde
{u}_j(t,x)-u(t,x) \bigr\rrvert \stackrel{\sf P} {\to} 0,\quad j\to
\infty.
\end{equation}
\end{thm}

\begin{proof} We use the notation from the proof of Theorem~\ref{thnabf}.
As in~\eqref{eqxijf}, we get
\[
\tilde{\xi}_j(t)\le C \int_0^t
\tilde{\xi}_j(s)\,ds+\tilde{\eta }_{\delta j},\quad t\in[0,1-
\delta],
\]
where
\begin{equation*}
\begin{split} \tilde{\xi}_j(t) & =\sup
_{x\in\mathbb{R}}\big|u(t,x)-\tilde{u}_j(t,x)\big|,
\\
\tilde{\eta}_{\delta j} & = \frac{1}{2a}\sup_{x\in\mathbb{R},\ t\in
[0,1-\delta]} \left|
\int_{(0,t]}g(t,x,s)\,d\mu(s) -\int_{(0,t]}
\tilde {S}_j(s)g(t,x,s)\,ds \right|. \end{split} %
\end{equation*}
The Gronwall's inequality implies that
\begin{equation}
\label{eqestt} \tilde{\xi}_j(t)\le C\tilde{\eta}_{\delta j},
\quad t\in[0,1-\delta].
\end{equation}
We will estimate $\tilde{\eta}_{\delta j}$. Consider
\begin{equation*}
\begin{split} & \int_{(0,t]}g(t,x,s)\,d\mu(s) -\int
_{(0,t]}\tilde{S}_j(s)g(t,x,s) \, ds
\\
&\quad =\frac{1}{j+1}\sum_{0\le k\le j} \biggl(\int
_{(0,t]}g(t,x,s)\,d\mu(s) -\int_{(0,t]}
{S}_k(s)g(t,x,s)\,ds \biggr). \end{split} %
\end{equation*}

Similarly to the estimates of $\eta_{\delta j}$ in Theorem~\ref
{thnabf}, we have
\begin{equation*}
\begin{split} \tilde{\eta}_{\delta j} & \le\frac{1}{j+1}\sum
_{0\le k\le j}(A_{\delta
}+B_{\delta k})
\\
&\quad +\sup_{t,x} \left|\int_{(0,1]}
\tilde{S}_j(s)g_{\delta}(t,x,s)\,ds-\int_{(0,t]}
\tilde{S}_j(s)g(t,x,s)\,ds \right|. \end{split} %
\end{equation*}
As in~\eqref{mlcdej}, we obtain
\begin{equation*}
\begin{split} &\sup_{t,x} \left|\int_{(0,1]}
\tilde{S}_j(s)g_{\delta}(t,x,s)\,ds-\int_{(0,t]}
\tilde{S}_j(s)g(t,x,s)\,ds \right|
\\
&\quad\le C \left|\int_{(1-\delta,1]}\tilde{S}_j(s)
\frac{s+\delta-1}{\delta
}\,ds \right|. \end{split} %
\end{equation*}

Taking the sum of terms \eqref{eqhdej}, we get
\[
\int_{(1-\delta,1]}\tilde{S}_j(s)\frac{s+\delta-1}{\delta}\,ds=
\int_{(0,1]}\tilde{h}_{\delta j}(s)\,d\mu(s).
\]
Here the functions
\[
\tilde{h}_{\delta j}(s)=\frac{1}{j+1}\sum_{0\le k\le j}
{h}_{\delta k}(s)
\]
are the Fej\`{e}r sums of function $h_{\delta}(s)=\frac{s+\delta
-1}{\delta}\mathbf{1}_{[1-\delta,1]}$. By well-known properties, $\tilde
{h}_{\delta j}(s)$ are uniformly bounded for every $\delta$ and $\tilde
{h}_{\delta j}(s)\to\tilde{h}_{\delta}(s)$ for every $s\in(0,1)$. Therefore,
\[
\int_{(0,1]}\tilde{h}_{\delta j}(s)\,d\mu(s)\stackrel{\sf
P} {\to} \int_{(0,1]}\tilde{h}_{\delta}(s)\,d\mu(s), \quad
j\to\infty,
\]
(here we used the condition $\mu(\{1\})=0$ a. s.). It remains to repeat
the reasoning in the proof of Theorem~\ref{thnabf} carried out
after~\eqref{eqhdej}.
\end{proof}

\section{Example}\label{scexam}

We obtained that solution of~\eqref{mweq1} is approximated by solutions
of~\eqref{mweq1nf} and~\eqref{mweq1nfe}. Equations~\eqref{mweq1nf}
and~\eqref{mweq1nfe} may be considered as nonstochastic for each fixed
$\omega$, properties of the solutions $u_j$ and $\tilde{u}_j$ follows
from the theory of deterministic wave equation.

Also, in some cases the rate of convergence in~\eqref{eqsupu} and~\eqref
{eqsupuf} may be estimated. By \eqref{eqeste} and \eqref{eqestt}, we
need to estimate $\eta_{\delta j}$ and $\tilde{\eta}_{\delta j}$ respectively.

As an example, consider SM $\mu$ given by \eqref{eqmuet} and \eqref
{eqmutp} provided that condition \eqref{eqconf} holds. Assumption~A6 is
fulfilled in this case because $\mu_t$ has a continuous version, A1--A5
are assumed as before. Recall that if~\eqref{eqlpmu} is fulfilled
for SM $\zeta$ in~\eqref{eqmuet} then it holds for $\mu$ (see
Remark~\ref{rmtone}).

In addition, assume that for some $L>0$, $\gamma>1/2$ and all $t,y_1,y_2$
\[
\left|\dfrac{\partial h(t, y_1)}{\partial t}-\dfrac{\partial h(t,
y_2)}{\partial t} \right|\le L|y_1-y_2|^\gamma.
\]
Then from Theorems~1 and 2~\cite{radpro} we obtain that we can change
the order of integration in~\eqref{eqmuet}, and
\begin{equation}
\label{eqmuah} \mu(A)=\int_{A}\,dt \int_{[a,b]}
\dfrac{\partial h(t, y)}{\partial t}\, d\zeta(y).
\end{equation}
The Fourier coefficients of $\mu$ are
\begin{equation*}
\begin{split} \xi_k&=\int_{(0,1]}\exp
\{-2\pi i k t \}\,d\mu(t) =\int_{(0,1]}\exp \{-2\pi i k t \}\,dt
\int_{[a,b]} \dfrac
{\partial h(t, y)}{\partial t}\,d\zeta(y)
\\
&\stackrel{(*)} {=}\int_{[a,b]} \,d\zeta(y) \int
_{(0,1]}\exp \{-2\pi i k t \}\dfrac{\partial h(t, y)}{\partial t}\,dt=\int
_{[a,b]} c_k(y)\,d\zeta(y) \end{split} %
\end{equation*}
where in (*) we again use Theorems~1 and 2~\cite{radpro}, $c_k(y)$
denotes the Fourier series coefficient of $\frac{\partial h(\cdot,
y)}{\partial t}$.

Therefore, the partial Fourier sums and Fej\`{e}r sums for this SM are
\begin{equation}
\label{eqsjtx} S_{j}(t)=\int_{[a,b]}
S_{j}^{(h)}(t,y)\,d\zeta(y),\quad\tilde {S}_{j}(t)=
\int_{[a,b]} \tilde{S}_{j}^{(h)}(t,y)\,d
\zeta(y),
\end{equation}
where $S_{j}^{(h)}(\cdot,y)$ and $\tilde{S}_{j}^{(h)}(\cdot,y)$ are
respectively Fourier and Fej\`{e}r sums of function $\frac{\partial
h(\cdot, y)}{\partial t}$ for each fixed $y$.
To estimate $\eta_{\delta j}$, consider
\begin{equation*}
\begin{split} &\int_{(0,t]}g(t,x,s)\,d\mu(s) -\int
_{(0,t]} S_j(s)g(t,x,s)\,ds
\\
&\stackrel{\eqref{eqmuah},\eqref{eqsjtx}} {=}\int_{(0,t]}
g(t,x,s)\, ds \int_{[a,b]} \dfrac{\partial h(s, y)}{\partial s}\,d\zeta(y)
\\
&\qquad- \int_{(0,t]} g(t,x,s) \,ds \int_{[a,b]}
S_{j}^{(h)}(s,y)\,d\zeta(y)
\\
&\quad=\int_{(0,t]}g(t,x,s)\, ds \int_{[a,b]}
\biggl(\dfrac{\partial h(s,
y)}{\partial s}-S_{j}^{(h)}(s,y) \biggr)\,d\zeta(y).
\end{split} %
\end{equation*}
The integral with respect to $\zeta$ may be estimated by~\eqref{eqestint}.
For the value of
\[
\dfrac{\partial h(s, y)}{\partial s}-S_{j}^{(h)}(s,y)
\]
we can find numerous results in the theory of classical Fourier series.
For example, if $h$ is smooth enough, we obtain $O(j^{-1}\ln j)$, see
\cite[Theorem (10.8) of Chapter II]{zygmund}. The detailed calculations
is not the subject of this paper.

Analogous considerations may be carried out for Fej\`{e}r sums, $\tilde
{\eta}_{\delta j}$, and $\tilde{S}_{j}^{(h)}(t,x)$.





\begin{thebibliography}{22}

\bibitem{BodnarchukWaveInR}
%
\begin{barticle}
\bauthor{\bsnm{Bodnarchuk}, \binits{I.}}:
\batitle{Mild solution of the wave equation with a general random measure}.
\bjtitle{Visnyk Kyiv University. Mathematics and Mechanics (in Ukrainian)}
\bvolume{24},
\bfpage{28}--\blpage{33}
(\byear{2010})
\end{barticle}
%
%
\OrigBibText
%
\begin{barticle}
\bauthor{\bsnm{Bodnarchuk}, \binits{I.}}:
\batitle{Mild solution of the wave equation with a general random measure}.
\bjtitle{Visnyk Kyiv University. Mathematics and Mechanics (in Ukrainian)}
\bvolume{24},
\bfpage{28}--\blpage{33}
(\byear{2010})
\end{barticle}
%
\endOrigBibText
\bptok{structpyb}%
\endbibitem

\bibitem{Bodnarchuk}
%
\begin{barticle}
\bauthor{\bsnm{Bodnarchuk}, \binits{I.M.}}:
\batitle{Wave equation with a stochastic measure}.
\bjtitle{Theory Probab. Math. Statist.}
\bvolume{94},
\bfpage{1}--\blpage{16}
(\byear{2017})
\bid{doi={10.1090/tpms/1005}, mr={3553450}}
\end{barticle}
%
%
\OrigBibText
%
\begin{barticle}
\bauthor{\bsnm{Bodnarchuk}, \binits{I.M.}}:
\batitle{Wave equation with a stochastic measure}.
\bjtitle{Theory Probab. Math. Statist.}
\bvolume{94},
\bfpage{1}--\blpage{16}
(\byear{2017})
\end{barticle}
%
\endOrigBibText
\bptok{structpyb}%
\endbibitem

\bibitem{DelgadoVencesSanzSole}
%
\begin{barticle}
\bauthor{\bsnm{Delgado-Vences}, \binits{F.J.}},
\bauthor{\bsnm{Sanz-Sol\'{e}}, \binits{M.}}:
\batitle{Approximation of a stochastic wave equation in dimension
three, with application to a support theorem in {H}\"{o}lder norm}.
\bjtitle{Bernoulli}
\bvolume{20},
\bfpage{2169}--\blpage{2216}
(\byear{2014})
\bid{doi={10.3150/13-BEJ554}, mr={3263102}}
\end{barticle}
%
%
\OrigBibText
%
\begin{barticle}
\bauthor{\bsnm{Delgado-Vences}, \binits{F.J.}},
\bauthor{\bsnm{Sanz-Sol\'{e}}, \binits{M.}}:
\batitle{Approximation of a stochastic wave equation in dimension
three, with application to a support theorem in {H}\"{o}lder norm}.
\bjtitle{Bernoulli}
\bvolume{20},
\bfpage{2169}--\blpage{2216}
(\byear{2014})
\end{barticle}
%
\endOrigBibText
\bptok{structpyb}%
\endbibitem

\bibitem{DelgadoVencesSanzSoleII}
%
\begin{barticle}
\bauthor{\bsnm{Delgado-Vences}, \binits{F.J.}},
\bauthor{\bsnm{Sanz-Sol\'{e}}, \binits{M.}}:
\batitle{Approximation of a stochastic wave equation in dimension three,
with application to a support theorem in {H}\"{o}lder norm: The non-stationary case}.
\bjtitle{Bernoulli}
\bvolume{22},
\bfpage{1572}--\blpage{1597}
(\byear{2016})
\bid{doi={10.3150/15-BEJ704}, mr={3474826}}
\end{barticle}
%
%
\OrigBibText
%
\begin{barticle}
\bauthor{\bsnm{Delgado-Vences}, \binits{F.J.}},
\bauthor{\bsnm{Sanz-Sol\'{e}}, \binits{M.}}:
\batitle{Approximation of a stochastic wave equation in dimension
three, with
application to a support theorem in {H}\"{o}lder norm: The non-stationary
case}.
\bjtitle{Bernoulli}
\bvolume{22},
\bfpage{1572}--\blpage{1597}
(\byear{2016})
\end{barticle}
%
\endOrigBibText
\bptok{structpyb}%
\endbibitem

\bibitem{dretop}
%
\begin{barticle}
\bauthor{\bsnm{Drewnowski}, \binits{L.}}:
\batitle{Topological rings of sets, continuous set functions, integration.~{III}}.
\bjtitle{Bull. Acad. Pol. Sci. S\'{e}r. sci. math. astron. phys.}
\bvolume{20},
\bfpage{439}--\blpage{445}
(\byear{1972})
\bid{mr={0316653}}
\end{barticle}
%
%
\OrigBibText
%
\begin{barticle}
\bauthor{\bsnm{Drewnowski}, \binits{L.}}:
\batitle{Topological rings of sets, continuous set functions,
integration.~{III}}.
\bjtitle{Bull. Acad. Pol. Sci. S\'{e}r. sci. math. astron. phys.}
\bvolume{20},
\bfpage{439}--\blpage{445}
(\byear{1972})
\end{barticle}
%
\endOrigBibText
\bptok{structpyb}%
\endbibitem

\bibitem{Kwapien}
%
\begin{bbook}
\bauthor{\bsnm{Kwapie\'{n}}, \binits{S.}},
\bauthor{\bsnm{Woyczy\'{n}ski}, \binits{W.A.}}:
\bbtitle{Random Series and Stochastic Integrals: Single and Multiple}.
\bpublisher{Birkh\"{a}user},
\blocation{Boston}
(\byear{1992})
\bid{doi={10.1007/978-1-4612-\\0425-1}, mr={1167198}}
\end{bbook}
%
%
\OrigBibText
%
\begin{bbook}
\bauthor{\bsnm{Kwapie\'{n}}, \binits{S.}},
\bauthor{\bsnm{Woyczy\'{n}ski}, \binits{W.A.}}:
\bbtitle{Random Series and Stochastic Integrals: Single and Multiple}.
\bpublisher{Birkh\"{a}user},
\blocation{Boston}
(\byear{1992})
\end{bbook}
%
\endOrigBibText
\bptok{structpyb}%
\endbibitem

\bibitem{tudor07}
%
\begin{barticle}
\bauthor{\bsnm{Maejima}, \binits{M.}},
\bauthor{\bsnm{Tudor}, \binits{C.A.}}:
\batitle{Wiener integrals with respect to the {H}ermite process and a non-central limit theorem}.
\bjtitle{Stochastic Analysis and Applications}
\bvolume{25}(\bissue{5}),
\bfpage{1043}--\blpage{1056}
(\byear{2007})
\bid{doi={10.1080/07362990701540519}, mr={2352951}}
\end{barticle}
%
%
\OrigBibText
%
\begin{barticle}
\bauthor{\bsnm{Maejima}, \binits{M.}},
\bauthor{\bsnm{Tudor}, \binits{C.A.}}:
\batitle{Wiener integrals with respect to the {H}ermite process and a
non-central limit theorem}.
\bjtitle{Stochastic Analysis and Applications}
\bvolume{25}(\bissue{5}),
\bfpage{1043}--\blpage{1056}
(\byear{2007})
\end{barticle}
%
\endOrigBibText
\bptok{structpyb}%
\endbibitem

\bibitem{memiva}
%
\begin{barticle}
\bauthor{\bsnm{M\'{e}min}, \binits{J.}},
\bauthor{\bsnm{Mishura}, \binits{Y.}},
\bauthor{\bsnm{Valkeila}, \binits{E.}}:
\batitle{Inequalities for the moments of {W}iener integrals with
respect to a fractional {B}rownian motion}.
\bjtitle{Statist. Probab. Lett.}
\bvolume{51},
\bfpage{197}--\blpage{206}
(\byear{2001})
\bid{doi={10.1016/S0167-7152(00)00157-7}, mr={1822771}}
\end{barticle}
%
%
\OrigBibText
%
\begin{barticle}
\bauthor{\bsnm{M\'{e}min}, \binits{J.}},
\bauthor{\bsnm{Mishura}, \binits{Y.}},
\bauthor{\bsnm{Valkeila}, \binits{E.}}:
\batitle{Inequalities for the moments of {W}iener integrals with
respect to a
fractional {B}rownian motion}.
\bjtitle{Statist. Probab. Lett.}
\bvolume{51},
\bfpage{197}--\blpage{206}
(\byear{2001})
\end{barticle}
%
\endOrigBibText
\bptok{structpyb}%
\endbibitem

\bibitem{PryharaShevchenkoWeqR2}
%
\begin{barticle}
\bauthor{\bsnm{Pryhara}, \binits{L.}},
\bauthor{\bsnm{Shevchenko}, \binits{G.}}:
\batitle{Stochastic wave equation in a plane driven by spatial stable noise}.
\bjtitle{Mod. Stoch. Theory Appl.}
\bvolume{3},
\bfpage{237}--\blpage{248}
(\byear{2016})
\bid{doi={\\10.15559/16-VMSTA62}, mr={3576308}}
\end{barticle}
%
%
\OrigBibText
%
\begin{barticle}
\bauthor{\bsnm{Pryhara}, \binits{L.}},
\bauthor{\bsnm{Shevchenko}, \binits{G.}}:
\batitle{Stochastic wave equation in a plane driven by spatial stable noise}.
\bjtitle{Mod. Stoch. Theory Appl.}
\bvolume{3},
\bfpage{237}--\blpage{248}
(\byear{2016})
\end{barticle}
%
\endOrigBibText
\bptok{structpyb}%
\endbibitem

\bibitem{PryharaShevchenkoWeqR3}
%
\begin{barticle}
\bauthor{\bsnm{Pryhara}, \binits{L.}},
\bauthor{\bsnm{Shevchenko}, \binits{G.}}:
\batitle{Wave equation with stable noise}.
\bjtitle{Teor. Imovir. Matem. Statist.}
\bvolume{96},
\bfpage{142}--\blpage{154}
(\byear{2017})
\bid{mr={3666878}}
\end{barticle}
%
%
\OrigBibText
%
\begin{barticle}
\bauthor{\bsnm{Pryhara}, \binits{L.}},
\bauthor{\bsnm{Shevchenko}, \binits{G.}}:
\batitle{Wave equation with stable noise}.
\bjtitle{Teor. Imovir. Matem. Statist.}
\bvolume{96},
\bfpage{142}--\blpage{154}
(\byear{2017})
\end{barticle}
%
\endOrigBibText
\bptok{structpyb}%
\endbibitem

\bibitem{RadchenkoInt}
%
\begin{bbook}
\bauthor{\bsnm{Radchenko}, \binits{V.}}:
\bbtitle{Integrals with Respect to General Stochastic Measures}.
\bpublisher{Proceedings of Institute of Mathematics, National Academy of
Science of Ukraine (in Russian)},
\blocation{Kyiv}
(\byear{1999})
\end{bbook}
%
%
\OrigBibText
%
\begin{bbook}
\bauthor{\bsnm{Radchenko}, \binits{V.}}:
\bbtitle{Integrals with Respect to General Stochastic Measures}.
\bpublisher{Proceedings of Institute of Mathematics, National Academy of
Science of Ukraine (in Russian)},
\blocation{Kyiv}
(\byear{1999})
\end{bbook}
%
\endOrigBibText
\bptok{structpyb}%
\endbibitem

\bibitem{radpro}
%
\begin{barticle}
\bauthor{\bsnm{Radchenko}, \binits{V.}}:
\batitle{On the product of a random and a real measure}.
\bjtitle{Theory Probab. Math. Statist.}
\bvolume{70},
\bfpage{197}--\blpage{206}
(\byear{2005})
\bid{doi={10.1090/S0094-9000-05-\\00639-3}, mr={2110872}}
\end{barticle}
%
%
\OrigBibText
%
\begin{barticle}
\bauthor{\bsnm{Radchenko}, \binits{V.}}:
\batitle{On the product of a random and a real measure}.
\bjtitle{Theory Probab. Math. Statist.}
\bvolume{70},
\bfpage{197}--\blpage{206}
(\byear{2005})
\end{barticle}
%
\endOrigBibText
\bptok{structpyb}%
\endbibitem

\bibitem{radt09}
%
\begin{barticle}
\bauthor{\bsnm{Radchenko}, \binits{V.}}:
\batitle{Sample functions of stochastic measures and {B}esov spaces}.
\bjtitle{Theory Probab. Appl.}
\bvolume{54},
\bfpage{160}--\blpage{168}
(\byear{2010})
\bid{doi={10.1137/\\S0040585X97984048}, mr={2766653}}
\end{barticle}
%
%
\OrigBibText
%
\begin{barticle}
\bauthor{\bsnm{Radchenko}, \binits{V.}}:
\batitle{Sample functions of stochastic measures and {B}esov spaces}.
\bjtitle{Theory Probab. Appl.}
\bvolume{54},
\bfpage{160}--\blpage{168}
(\byear{2010})
\end{barticle}
%
\endOrigBibText
\bptok{structpyb}%
\endbibitem

\bibitem{RadchenkoHilbertSpace}
%
\begin{barticle}
\bauthor{\bsnm{Radchenko}, \binits{V.}}:
\batitle{Evolution equations driven by general stochastic measures in {H}ilbert space}.
\bjtitle{Theory Probab. Appl.}
\bvolume{59},
\bfpage{328}--\blpage{339}
(\byear{2015})
\bid{doi={10.1137/\\S0040585X97T987119}, mr={3416054}}
\end{barticle}
%
%
\OrigBibText
%
\begin{barticle}
\bauthor{\bsnm{Radchenko}, \binits{V.}}:
\batitle{Evolution equations driven by general stochastic measures in {H}ilbert space}.
\bjtitle{Theory Probab. Appl.}
\bvolume{59},
\bfpage{328}--\blpage{339}
(\byear{2015})
\end{barticle}
%
\endOrigBibText
\bptok{structpyb}%
\endbibitem

\bibitem{radt18}
%
\begin{barticle}
\bauthor{\bsnm{Radchenko}, \binits{V.}}:
\batitle{Fourier series expansion of stochastic measures}.
\bjtitle{Theory Probab. Appl.}
\bvolume{63},
\bfpage{389}--\blpage{401}
(\byear{2018})
\bid{mr={3796494}}
\end{barticle}
%
%
\OrigBibText
%
\begin{barticle}
\bauthor{\bsnm{Radchenko}, \binits{V.}}:
\batitle{Fourier series expansion of stochastic measures}.
\bjtitle{Theory Probab. Appl.}
\bvolume{63},
\bfpage{389}--\blpage{401}
(\byear{2018})
\end{barticle}
%
\endOrigBibText
\bptok{structpyb}%
\endbibitem

\bibitem{radt06}
%
\begin{barticle}
\bauthor{\bsnm{Radchenko}, \binits{V.M.}}:
\batitle{Parameter-dependent integrals with respect to general random measures}.
\bjtitle{Theor. Probability and Math. Statist.}
\bvolume{75},
\bfpage{161}--\blpage{165}
(\byear{2007})
\bid{doi={10.1090/S0094-9000-08-00722-9}, mr={2321189}}
\end{barticle}
%
%
\OrigBibText
%
\begin{barticle}
\bauthor{\bsnm{Radchenko}, \binits{V.M.}}:
\batitle{Parameter-dependent integrals with respect to general random measures}.
\bjtitle{Theor. Probability and Math. Statist.}
\bvolume{75},
\bfpage{161}--\blpage{165}
(\byear{2007})
\end{barticle}
%
\endOrigBibText
\bptok{structpyb}%
\endbibitem

\bibitem{RadStef}
%
\begin{barticle}
\bauthor{\bsnm{Radchenko}, \binits{V.M.}},
\bauthor{\bsnm{Stefans'ka}, \binits{N.O.}}:
\batitle{Fourier and {F}ourier--{H}aar series for stochastic measures}.
\bjtitle{Teor. Imovir. Matem. Statist.}
\bvolume{96},
\bfpage{155}--\blpage{162}
(\byear{2017})
\end{barticle}
%
%
\OrigBibText
%
\begin{barticle}
\bauthor{\bsnm{Radchenko}, \binits{V.M.}},
\bauthor{\bsnm{Stefans'ka}, \binits{N.O.}}:
\batitle{Fourier and {F}ourier--{H}aar series for stochastic measures}.
\bjtitle{Teor. Imovir. Matem. Statist.}
\bvolume{96},
\bfpage{155}--\blpage{162}
(\byear{2017})
\end{barticle}
%
\endOrigBibText
\bptok{structpyb}%
\endbibitem

\bibitem{rad16}
%
\begin{barticle}
\bauthor{\bsnm{Radchenko}, \binits{V.M.}},
\bauthor{\bsnm{Stefans'ka}, \binits{N.O.}}:
\batitle{Fourier transform of general stochastic measures}.
\bjtitle{Theory Probab. Math. Statist.}
\bvolume{94},
\bfpage{151}--\blpage{158}
(\byear{2017})
\bid{doi={10.1090/\\tpms/1015}, mr={3553460}}
\end{barticle}
%
%
\OrigBibText
%
\begin{barticle}
\bauthor{\bsnm{Radchenko}, \binits{V.M.}},
\bauthor{\bsnm{Stefans'ka}, \binits{N.O.}}:
\batitle{Fourier transform of general stochastic measures}.
\bjtitle{Theory Probab. Math. Statist.}
\bvolume{94},
\bfpage{151}--\blpage{158}
(\byear{2017})
\end{barticle}
%
\endOrigBibText
\bptok{structpyb}%
\endbibitem

\bibitem{samtaq}
%
\begin{bbook}
\bauthor{\bsnm{Samorodnitsky}, \binits{G.}},
\bauthor{\bsnm{Taqqu}, \binits{M.}}:
\bbtitle{Stable Non-Gaussian Random Processes}.
\bpublisher{Chapman and Hall},
\blocation{London}
(\byear{1994})
\bid{mr={1280932}}
\end{bbook}
%
%
\OrigBibText
%
\begin{bbook}
\bauthor{\bsnm{Samorodnitsky}, \binits{G.}},
\bauthor{\bsnm{Taqqu}, \binits{M.}}:
\bbtitle{Stable Non-Gaussian Random Processes}.
\bpublisher{Chapman and Hall},
\blocation{London}
(\byear{1994})
\end{bbook}
%
\endOrigBibText
\bptok{structpyb}%
\endbibitem

\bibitem{tudor09}
%
\begin{barticle}
\bauthor{\bsnm{Tudor}, \binits{C.}}:
\batitle{On the {W}iener integral with respect to a sub-fractional {B}rownian
motion on an interval}.
\bjtitle{Journal of Mathematical Analysis and Applications}
\bvolume{351}(\bissue{1}),
\bfpage{456}--\blpage{468}
(\byear{2009})
\bid{doi={10.1016/j.jmaa.2008.10.041}, mr={2472957}}
\end{barticle}
%
%
\OrigBibText
%
\begin{barticle}
\bauthor{\bsnm{Tudor}, \binits{C.}}:
\batitle{On the {W}iener integral with respect to a sub-fractional {B}rownian
motion on an interval}.
\bjtitle{Journal of Mathematical Analysis and Applications}
\bvolume{351}(\bissue{1}),
\bfpage{456}--\blpage{468}
(\byear{2009})
\end{barticle}
%
\endOrigBibText
\bptok{structpyb}%
\endbibitem

\bibitem{vahtar}
%
\begin{bbook}
\bauthor{\bsnm{Vakhania}, \binits{N.N.}},
\bauthor{\bsnm{Tarieladze}, \binits{V.I.}},
\bauthor{\bsnm{Chobanian}, \binits{S.A.}}:
\bbtitle{Probability Distributions on Banach Spaces}.
\bpublisher{D. Reidel Publishing Co.},
\blocation{Dordrecht}
(\byear{1987})
\bid{doi={\\10.1007/978-94-009-3873-1}, mr={1435288}}
\end{bbook}
%
%
\OrigBibText
%
\begin{bbook}
\bauthor{\bsnm{Vakhania}, \binits{N.N.}},
\bauthor{\bsnm{Tarieladze}, \binits{V.I.}},
\bauthor{\bsnm{Chobanian}, \binits{S.A.}}:
\bbtitle{Probability Distributions on Banach Spaces}.
\bpublisher{D. Reidel Publishing Co.},
\blocation{Dordrecht}
(\byear{1987})
\end{bbook}
%
\endOrigBibText
\bptok{structpyb}%
\endbibitem

\bibitem{zygmund}
%
\begin{bbook}
\bauthor{\bsnm{Zygmund}, \binits{A.}}:
\bbtitle{Trigonometric Series}.
\bpublisher{Cambridge University Press},
\blocation{Cambridge}
(\byear{2002})
\bid{mr={0236587}}
\end{bbook}
%
%
\OrigBibText
%
\begin{bbook}
\bauthor{\bsnm{Zygmund}, \binits{A.}}:
\bbtitle{Trigonometric Series}.
\bpublisher{Cambridge University Press},
\blocation{Cambridge}
(\byear{2002})
\end{bbook}
%
\endOrigBibText
\bptok{structpyb}%
\endbibitem

\end{thebibliography}
\end{document}